\documentclass[a4paper,11pt,reqno]{amsart}
\usepackage{amsmath}
\usepackage{amsthm,enumerate}

\usepackage{graphicx}
\usepackage{amssymb}

\usepackage{appendix}

\usepackage[italian,english]{babel}

\usepackage[utf8x]{inputenc}
\usepackage{fancyhdr}

\usepackage{calc}
\usepackage{url}

\usepackage[text={6.3in,8.6in},centering]{geometry}

\usepackage{srcltx, inputenc}

\usepackage[T1]{fontenc}

\usepackage[usenames,dvipsnames]{color}

\makeatletter
\def\cleardoublepage{\clearpage\if@twoside \ifodd\c@page\else%
         \hbox{}%
     \thispagestyle{empty}%              % Empty header styles
     \newpage%
     \if@twocolumn\hbox{}\newpage\fi\fi\fi}
\makeatother

\let\cleardoublepage\clearpage

\addto\captionsitalian{}

\newtheorem{thm}{Theorem}[section]
\newtheorem{cor}[thm]{Corollary}

\newtheorem{pro}[thm]{Proposition}
\newtheorem{den}[thm]{Definition}
\newtheorem{oss}[thm]{Remark}
\numberwithin{equation}{section}

\begin{document}

\title[]{Nonlinear characterizations\\ of stochastic completeness}

\author{Gabriele Grillo}
\address{\hbox{\parbox{5.7in}{\medskip\noindent{Dipartimento di Matematica,\\
Politecnico di Milano,\\
   Piazza Leonardo da Vinci 32, 20133 Milano, Italy.
   \\[3pt]
        \em{E-mail address: }{\tt
          gabriele.grillo@polimi.it%\\ \it Corresponding author
          }}}}}

\author{Kazuhiro Ishige}
\address{\hbox{\parbox{5.7in}{\medskip\noindent{Graduate School of Mathematical Sciences,\\
The University of Tokyo,\\
  3-8-1 Komaba, Meguro-ku, Tokyo 153-8914, Japan. \\[3pt]
        \em{E-mail address: }{\tt
          ishige@ms.u-tokyo.ac.jp}}}}}

\author[Matteo Muratori]{Matteo Muratori}
\address{\hbox{\parbox{5.7in}{\medskip\noindent{Dipartimento di Matematica,\\
Politecnico di Milano,\\
   Piazza Leonardo da Vinci 32, 20133 Milano, Italy. \\[3pt]
        \em{E-mail address: }{\tt
          matteo.muratori@polimi.it}}}}}

%\footnotetext{Corresponding author}
%
%\address {Gabriele Grillo: Dipartimento di Matematica, Politecnico di Milano, Piaz\-za Leonardo da Vinci 32, 20133 Milano, Italy}
%\email{gabriele.grillo@polimi.it}
%
%\address{Matteo Muratori: Dipartimento di Matematica, Politecnico di Milano, Piaz\-za Leonardo da Vinci 32, 20133 Milano, Italy}
%\email{matteo.muratori@polimi.it}
%
%\address{Juan Luis V\'azquez: Departamento de Matem\'aticas, Universidad Aut\'onoma de Madrid, 28049 Madrid, Spain}
%\email{juanluis.vazquez@uam.es}
%
%\subjclass[2010]{Primary: 35R01. Secondary: 35K65, 58J35, 35A01, 35A02, 35B44}
%\keywords{Porous medium equation; Cartan-Hadamard manifolds;  very negative curvature; separable solutions; asymptotics; nonlinear elliptic equations.}

%
%\address {Gabriele Grillo, Matteo Muratori: Dipartimento di Matematica, Politecnico di Milano, Piaz\-za Leonardo da Vinci 32, 20133 Milano, Italy}
%
%\email {gabriele.grillo@polimi.it}
%
%
%\email {matteo.muratori@polimi.it}

%\vskip-20pt



\maketitle

\begin{abstract}
We prove that conservation of probability for the free heat semigroup on a Riemannian manifold $M$ (namely stochastic completeness), hence a linear property, is equivalent to uniqueness of positive, bounded solutions to nonlinear evolution equations of fast diffusion type on $M$ of the form $u_t=\Delta \phi(u)$, $\phi$ being an arbitrary concave, increasing positive function, regular outside the origin and with $\phi(0)=0$. Either property is also equivalent to nonexistence of nonnegative, nontrivial, bounded solutions to the elliptic equation $\Delta W=\phi^{-1}(W)$, with $\phi$ as above. As a consequence, explicit criteria for uniqueness or nonuniqueness of bounded solutions to fast diffusion-type equations on manifolds are given, these being the first results on such issues.
%%%%%%%%%%

\vskip8pt
\noindent\textsc{Résumé}. On montre que la conservation de la probabilité pour le semi-groupe libre de la chaleur sur une variété Riemannienne $M$ (c'est à dire le fait que $M$ soit complète au sens stochastique), donc une propriété purement linéaire, est équivalente à l'unicité de solutions bornées et positives d'équations d'évolution non linéaires de type diffusion rapide de la forme $u_t=\Delta \phi(u)$, où $\phi$ est une fonction arbitraire concave, croissante, positive, régulière en dehors de l'origine et satisfaisante $ \phi(0)=0 $. Les deux propriétés sont aussi équivalentes à la non existence de solutions bornées, positives et non triviales de l'équation elliptique $\Delta W=\phi^{-1}(W)$, où $\phi$ est comme ci-dessus. Par conséquent, on peut donner des critères explicites pour l'unicité ou la non unicité de solutions bornées d'équations de type diffusion rapide sur des variétés, ces derniers étant les premiers résultats concernant tels problèmes.
\end{abstract}

\section{Introduction}

Let $M$ be a connected, noncompact Riemannian manifold of dimension $N\ge2$. We recall that $M$ is said to be stochastically complete if the lifetime of Brownian paths on $M$ is a.s.~infinite or, in analytical terms, if the free heat semigroup preserves the identity in the sense that
\[
\int_Mp(t,x,y)\,{d}y=1
\]
for all $(x,t)\in M\times(0,+\infty)$, where $p$ is the (minimal) heat kernel on $M$ and ${d}y$ denotes the Riemannian measure on $M$. Note that $M$ need not be geodesically complete.

A number of analytic and/or geometric conditions ensuring stochastic completeness, or incompletess, of $M$, are by now available, and it would be hopeless to give a full list of the related literature, for which we refer to the comprehensive discussions provided by Grigor'yan e.g.~in \cite{GrigHK, Grig, GrigBams}, but we mention at least that either curvature conditions, volume growth properties and function theoretic conditions, possibly originating from different contexts but applicable to the present problem as well, can be successfully used to verify stochastic completeness, see e.g.~\cite{A, D, G, GrigTo, H, I1, I, IM, K, M, M2, PRS, PRS2, S, T, Y}. To give a flavour of the curvature conditions that guarantee stochastic completeness, one can restrict the attention to Cartan-Hadamard manifolds and notice that qualitatively, in such situation, stochastic completeness holds when curvature does not diverge to minus infinity faster than quadratically, and it does not hold otherwise (see e.g.~\cite{M} for a precise statement).

Of particular interest for the present discussion is the well-known fact that stochastic completeness is related, in fact equivalent, to two other analytic properties. In fact, having fixed two positive constants $\lambda, T$ and a \it bounded \rm function $u_0$, the following three properties turn out to be equivalent (see \cite[Theorem 8.18]{GrigHK}):
\begin{itemize}

\item $M$ is stochastically complete;
\item The equation $\Delta v=\lambda v$ does not admit any nonnegative, nontrivial, bounded solutions;
\item The Cauchy problem
\begin{equation}\label{heat}
\begin{cases}
u_t = \Delta u & \text{in} \ M \times[0,T] \, , \\
u(\cdot,0)=u_0 & \text{in} \ M \, ,
\end{cases}
\end{equation}
has a unique nonnegative solution in $L^\infty(M\times(0,T))$ for $u_0\in L^\infty(M)$ with $u_0\ge 0$.
\end{itemize}

Uniqueness, for bounded data, of solutions to the heat equation on Riemannian manifolds is therefore another form of stochastic completeness, and is in particular known to hold under some clear geometric assumptions, see e.g.~\cite{IM, M} or Corollary \ref{compl} below.

In the recent years, the issue of uniqueness for \it nonlinear \rm evolution equations posed on Riemannian manifolds has been the object of some investigation. In particular, the \it porous medium equation \rm (see \cite{V} for a thorough discussion of this equation in the Euclidean case)
\begin{equation}\label{pme}
\begin{cases}
u_t=\Delta u^m & \text{in } M\times[0,T] \, , \\
u(\cdot,0) = u_0  & \text{in } M \, ,
\end{cases}
\end{equation}
where one has to assume that $m>1$ so that the diffusion coefficient $mu^{m-1}$ vanishes when $u$ does, has been investigated e.g.~in \cite{GMP, GMP2, GMV}. Among other issues, uniqueness results for suitable classes of data, possibly requiring appropriate curvature bounds, have been proved in these papers.

Much less is known so far on \eqref{pme} in the \it singular \rm case $m<1$, which goes under the name of \it fast diffusion equation \rm since the diffusion coefficient $mu^{m-1}$ diverges  when $u$ vanishes, thus forcing infinite speed of propagation and, in fact, even the possibility that solutions vanish in finite time. This has been indeed proved to happen, for suitable data, when $M$ admits a spectral gap (i.e.~$\min \sigma(-\Delta)>0$, where $\sigma(-\Delta)$ is the $L^2$ spectrum of $-\Delta$), which is the case for instance on Cartan-Hadamard manifolds when Sec$\,\le-k$ for a suitable $k>0$ (see \cite{Mc}), and in particular on the hyperbolic space (see \cite{BGV}). It has to be noted that, in the Euclidean case $M=\mathbb{R}^N$, \emph{existence and uniqueness} of solutions to \eqref{pme} hold instead even when $u_0$ is merely required to be \it locally integrable\rm, see \cite{HP}. Little seems to be known in this connection on noncompact manifolds other than $\mathbb{R}^N$. Uniqueness for \it strong \rm solutions is presently known only when curvature is allowed to be negative, but decaying to zero at least quadratically, see \cite{BS}.

We shall establish here that equivalence between stochastic completeness, a concept a priori related to a \it linear \rm evolution on $M$, and uniqueness of nonnegative, bounded solutions to \eqref{pme} is still valid in the fast diffusion range $m<1$. We devote Section \ref{case-fde} to a deeper discussion of this quite significant case. In fact we shall prove a more general result, namely that the above mentioned equivalence holds if \eqref{pme} is replaced by
\begin{equation}\label{pmephi}\begin{cases}
u_t=\Delta \phi(u) & \text{in } M\times(0,T) \, , \\
u(\cdot, 0)=u_0 & \text{in } M \, ,
\end{cases}
\end{equation}
where $\phi$ is an \it arbitrary \rm concave, increasing positive \rm function, regular outside the origin and satisfying $\phi(0)=0$.  This provides, thanks to \cite[Theorem 8.18]{GrigHK}), what is in our view a rather unexpected connection between properties of \it linear \rm and \it nonlinear \rm evolutions on manifolds. As a byproduct, one gets immediately that uniqueness of nonnegative bounded solutions to the nonlinear evolution \eqref{pmephi} for a given concave, increasing, positive and regular outside the origin structure function $\phi$ (including the linear case $\phi(x)=x$) is \it equivalent \rm to the same property for the same problem associated with \it any \rm such structure function $\phi$, a fact that we find remarkable. We stress that, to our knowledge, these are the first results available in such a general framework.

Problem \eqref{pmephi} is deeply related to the semilinear elliptic equation
\begin{equation}\label{elliptic-phi}
\Delta W=\phi^{-1}(W) \qquad \text{in } M \, .
\end{equation}
Indeed, the existence of a nonnegative, nontrivial, bounded solution to \eqref{elliptic-phi} is also equivalent to stochastic incompleteness, or analogously, stochastic completeness is equivalent to the fact that $ W \equiv 0 $ is the only nonnegative, nontrivial, bounded solution to \eqref{elliptic-phi}. This elliptic result is a byproduct of our methods of proof (which in particular do not exploit any kind of continuity of solutions), even though in the specific case of the Laplace-Beltrami operator was basically already known: see Remark \ref{rem-ell} below.

%Thus existence of positive, bounded solutions to \eqref{ell} holds on any stochastically incomplete manifold, and in particular when curvature diverges to minus infinity faster than quadratically in the precise sense given in \cite{M}. Under such conditions, nonuniqueness of nonnegative, bounded solutions of \eqref{pmephi} also holds, in sharp contrast with the Euclidean situation.

%A similar property holds as concerns existence of bounded, positive, nontrivial solutions to the stationary problem \eqref{ell} associated with different structure functions $\phi$.

%The paper is organized as follows. In Section 2 we collect some geometric preliminaries. In Section 3 we state and prove our main result (Theorem ???). In section 4 we state and prove some corollaries on uniqueness and nonuniqueness of bounded solutions to \eqref{pmephi} and on existence or nonexistence of nonnegative, nontrivial bounded solutions to \eqref{ell}.

%\color{red} altra possibilita': mettere sia il risultato lineare che quello nonlineare qui. \normalcolor

\subsection{Statement of the main result}

We denote by $ \mathfrak{C} $ the class of all functions $ \phi: \mathbb{R}^+ \to \mathbb{R}^+ $ satisfying the following assumptions:
\begin{gather*} \label{g1}
\phi \ \text{is \emph{concave}} \, , \\ \label{g2}
\phi \ \text{is \emph{strictly increasing} with} \ \phi(0)=0 \, , \\ \label{g3}
\phi \in C(\mathbb{R}^+) \cap C^1(\mathbb{R}^+ \setminus \{ 0 \}) \, .
\end{gather*}

Let $ M $ be a connected, noncompact Riemannian manifold. For any $ T \in (0,+\infty] $, any nonnegative $ u_0 \in L^\infty(M) $ and any $ \phi \in \mathfrak{C} $, we can consider the nonlinear parabolic problem
\begin{equation}\label{gen-zero-bdd}
\begin{cases}
u_t = \Delta \phi(u) & \text{in } M \times (0,T) \, , \\
u(\cdot,0)=u_0  & \text{in } M  \, ,
\end{cases}
\end{equation}
which always has an $ L^\infty(M\times(0,T)) $ nonnegative solution, by standard results. As mentioned above, we shall also be concerned with the problem of existence of nonnegative, nontrivial, bounded solutions to the semilinear elliptic equation \eqref{elliptic-phi}. For the precise meaning of solution to \eqref{elliptic-phi} and \eqref{gen-zero-bdd}, we refer to Section \ref{preliminaries}.

\medskip

We are now ready to state the main result of this paper.

\begin{thm}\label{main-teo}
Let $ M $ be a connected, noncompact Riemannian manifold. Let $ \phi \in \mathfrak{C} $, $ T \in (0,+\infty] $ and $ u_0 \in L^\infty(M)$, with $ u_0 \ge 0 $. Then the following assertions are equivalent:
\begin{enumerate}[(a)]
\item $ M $ is stochastically incomplete. \label{a}
\item The Cauchy problem \eqref{gen-zero-bdd} admits at least two nonnegative solutions in $ L^\infty(M\times(0,T)) $. \label{b}
\item The semilinear elliptic equation \eqref{elliptic-phi} admits a nonnegative, nontrivial, bounded solution. \label{c}
\end{enumerate}
\end{thm}

Theorem~\ref{main-teo} is of course equivalent to the following one, that nevertheless we state separately for the reader's convenience.
\begin{thm}\label{main-teob}
Let $ M $ be a connected, noncompact Riemannian manifold. Let $ \phi \in \mathfrak{C} $, $ T \in (0,+\infty] $ and $ u_0 \in L^\infty(M)$, with $ u_0 \ge 0 $. Then the following assertions are equivalent:
\begin{enumerate}[(a)]
\item[(a')] $ M $ is stochastically complete.
\item[(b')] The Cauchy problem \eqref{gen-zero-bdd} has a unique nonnegative solution in $ L^\infty(M\times(0,T)) $.
\item[(c')] The semilinear elliptic equation \eqref{elliptic-phi} does not admit any nonnegative, nontrivial, bounded solutions. \label{c'}
\end{enumerate}
\end{thm}

As an immediate Corollary of Theorem~\ref{main-teo},
we have the following result, which follows by simply noting that $\phi$, $T$ and $u_0$ are arbitrary in the above statement.

\begin{cor}\label{cor1} Let $ M $ be a connected, noncompact Riemannian manifold. The following statements are equivalent, and each of them is equivalent to stochastic incompleteness (resp.~completeness):
\begin{itemize}
\item For some bounded datum $u_0\ge0$, some function $ \phi \in \mathfrak{C} $ and some $T\in(0,+\infty]$ the Cauchy problem \eqref{gen-zero-bdd} admits at least two nonnegative solutions (resp.~a unique nonnegative solution) in $ L^\infty(M\times(0,T)) $;
\item For all bounded data $u_0\ge0$, all function $ \phi \in \mathfrak{C} $ and all $T\in(0,+\infty]$ the Cauchy problem \eqref{gen-zero-bdd} admits at least two nonnegative solutions (resp.~a unique nonnegative solution) in $ L^\infty(M\times(0,T)) $;
\item For some function $ \phi \in \mathfrak{C} $ the semilinear elliptic equation \eqref{elliptic-phi} admits (resp.~does not admit) a nonnegative, nontrivial, bounded solution;
\item For all functions $ \phi \in \mathfrak{C} $ the semilinear elliptic equation \eqref{elliptic-phi} admits (resp.~does not admit) a nonnegative, nontrivial, bounded solution.
\end{itemize}

\end{cor}
%Of course, analogous equivalence results hold when stochastic \emph{completeness} is considered instead.

\medskip
\begin{oss}\label{rem-ell}
\rm Notice that $\phi(u)=u$ belongs to $ \mathfrak{C} $. Hence Corollary \ref{cor1} implies that uniqueness for solutions to the generalized fast diffusion equation \eqref{gen-zero-bdd} holds, for (nonnegative) bounded data and $\phi\in \mathfrak{C}$, if and only if uniqueness holds for solutions to the heat equation \eqref{heat} for bounded data. Similarly, existence of nonnegative, nontrivial, bounded solutions to the semilinear elliptic problem \eqref{elliptic-phi} with $\phi\in \mathfrak{C}$ holds if and only if the linear equation $\Delta v=\lambda v$ admits a nonnegative, nontrivial, bounded solution for some, hence all $\lambda>0$. Let us focus on the more general semilinear equation
\begin{equation}\label{elliptic-general}
\Delta W=f(W) \qquad \text{in } M \, .
\end{equation}
Actually, as a consequence of previous results, the elliptic equivalence theorem is even stronger: from \cite[Theorem 3.1 and condition (vii) at page 43]{PRS}, which in fact goes back to \cite{PRS2}, one deduces that $M$ is stochastically complete if and only if for every strictly increasing $ f \in C^0(\mathbb{R}^+)  $, with $ f(0)=0 $, equation \eqref{elliptic-general} does not admit any nonnegative, nontrivial, bounded solution. In other words, one can drop the convexity assumption on $ \phi^{-1} $. See also \cite{MV} for generalizations to suitable nonlinear diffusion operators. However, for the methods of \cite{PRS,PRS2} to work the continuity of $W$ is essential, which for the Laplace-Beltrami operator is for free due to elliptic regularity solutions are at least $C^{1,\alpha}$), but may not hold for more general operators: we refer to Remark \ref{rrr} below. Vice versa, if $ M $ is stochastically incomplete then a nonnegative, nontrivial solution to $ \Delta v = \lambda v $ with $ \| v \|_\infty=1 $ is always a subsolution to \eqref{elliptic-general}, provided $ \lambda $ is chosen to be the maximum of $ f^\prime $ in $ [0,1] $. This does not require convexity either, only $ f \in C^1(\mathbb{R}^+) $. Once a nonnegative, nontrivial, bounded subsolution exists then a bounded solution lying above the subsolution can be constructed.
\end{oss}

\smallskip
\begin{oss}\label{reg}\rm
The problem of optimal regularity of merely \emph{distributional} solutions to \eqref{gen-zero-bdd} is, to our knowledge, presently open, even in the Euclidean setting, although it is known under some additional assumptions (for instance local finiteness of the energy of solutions, see e.g.~\cite{DUV}).
\end{oss}

\smallskip

\begin{oss}\label{rrr}\rm
The methods used here are in principle applicable to more general second order differential operators, not only the Laplace-Beltrami one, being based essentially only on suitable comparison principles, see Appendix \ref{app}. We stress that for our arguments to work, we do not need any kind of continuity of solutions to \eqref{elliptic-phi} or \eqref{gen-zero-bdd}.
%: it is assumed at some point in Appendix \ref{app} and in the proof of Theorem \ref{main-teo} (implicitly), just to simplify the discussion there.
\end{oss}

\medskip

The paper is organized as follows. In Section \ref{case-fde} we explain part of our strategy by focusing on the very important model case $ \phi(u)=u^m $ with $ m \in (0,1) $. In Section \ref{preliminaries} we briefly recall the concept of distributional solution for the differential equations considered. In Section \ref{proofmain} we prove our main result in the form of Theorem \ref{main-teo}. Appendix \ref{app} contains concise proofs of the comparison principles used in our argument. Appendix \ref{cor} states, for the reader's convenience, a list of known geometric and analytic conditions that ensure that a manifold is stochastically complete or incomplete, that in particular give explicit criteria on $M$ ensuring uniqueness, or nonuniqueness, of solutions to \eqref{gen-zero-bdd}.

\section{The fast diffusion case}\label{case-fde}

As it has already been mentioned in the Introduction, Theorems \ref{main-teo}--\ref{main-teob} include as a particular case $ L^\infty $ uniqueness and nonuniqueness sharp results for the \emph{fast diffusion equation} \eqref{pme} on manifolds, which corresponds to $ \phi(u)=u^m $ for any $ m \in (0,1) $. Such equation has widely been studied on Euclidean space: we quote, without any claim to completeness, the references \cite{HP,PZ0,PZ,V,BBDGV,BDGV}. More recently, it has also been investigated on manifolds, mainly with negative curvature: see \cite{BGV,GM} and \cite[Section 4]{BS}.

More specifically, as concerns existence and uniqueness issues, in the celebrated paper \cite{HP} M.A.~Herrero and M.~Pierre set up a whole theory for data merely in $ L^1_{\mathrm{loc}} $. Indeed, \cite[Theorem 2.1]{HP} establishes that for any $ m \in (0,1) $ and $ u_0 \in L^1_{\mathrm{loc}}(\mathbb{R}^N) $ there exists a (distributional) solution $ u $ of \eqref{pme} (at least up to some finite $T>0$), which is moreover locally bounded for positive times if $ m $ is larger than the critical value $ (N-2)^+/N $ \cite[Theorem 2.2]{HP}. This is by itself a remarkable result, since it means that the behaviour at spatial infinity of $ u_0(x) $ does not count at all, in sharp contrast with the linear case, namely the heat equation (see also the Introduction to \cite{GMP} for a detailed discussion on these topics). Furthermore, uniqueness holds in such a wide class \cite[Theorem 2.3]{HP}, under the additional assumption that solutions are \emph{strong}, i.e.~that $u_t$ is an $ L^1_{\mathrm{loc}}(\mathbb{R}^N \times (0,T)) $ function.

Existence for data in $ L^1_{\mathrm{loc}} $ also holds on manifolds regardless of curvature or volume-growth assumptions, since the proof of \cite[Theorem 2.1]{HP} is purely local. On the other hand, an analogue of the Herrero-Pierre result was proved in \cite{BS} (see Theorem 4.9 there), provided the Ricci curvature is bounded from below by an inverse quadratic function of the distance from a reference point, i.e.~negative curvatures must vanish at least with a quadratic rate at infinity.

\smallskip

The results entailed by Theorem \ref{main-teo} show that the picture is totally different on a general manifold $M$: if the latter is stochastically incomplete then uniqueness to \eqref{pme} (for $ m \in (0,1) $) \emph{fails} even in the class of bounded solutions. In order to give a flavor of the methods of proof we shall exploit in Section \ref{proofmain}, let us construct an explicit solution to \eqref{pme} with $ u_0 \equiv 0 $ which is not identically zero, on any stochastically incomplete Riemannian manifold $M$. Indeed, thanks to \cite[Theorem 8.18]{GrigHK}, we know that there exists a nonnegative, nontrivial, bounded solution $v$ to
$$
\Delta v = v \quad \text{in } M \, , \qquad \| v \|_{L^\infty(M)}=1 \, .
$$
In particular,
$$
\Delta v \ge v^{\frac{1}{m}} \quad \text{in } M \, ,
$$
so that by a standard monotone procedure (see e.g.~the proof of $(\ref{b}) \Rightarrow (\ref{c})$ in Section \ref{proofmain}) one can find a nonnegative, nontrivial, bounded function $ W $ solving
$$
\Delta W = W^{\frac{1}{m}} \quad \text{in } M \, ,
$$
which is in fact larger than $v$. Hence a routine computation shows that the separable function
$$
u := \left[ (1-m) \,t \right]^{\frac{1}{1-m}} W^{\frac{1}{m}}
$$
is a solution to $ u_t=\Delta u^m $, bounded in $ M \times [0,T] $ for every $ T>0 $. Since $ u(\cdot,0) \equiv 0 $, this yields nonuniqueness for the Cauchy problem \eqref{pme} with $ u_0 \equiv 0 $. Then one could actually exploit the existence of such a solution to construct multiple solutions for \emph{any} initial datum: we refer to the techniques adopted in the proof of $(\ref{a}) \Rightarrow (\ref{b})$.

\section{On the concepts of solution}\label{preliminaries}

We collect here the basic definitions of (distributional or very weak) solutions to \eqref{gen-zero-bdd} and \eqref{elliptic-phi}, as well as sub/supersolutions. We shall also briefly discuss existence of solutions to \eqref{gen-zero-bdd} for bounded initial data, which is rather standard. At the end of the paper, for the reader's convenience, we provide a short Appendix where related comparison principles (on balls) are established, since the latter constitute a key tool of our analysis.

\begin{den}\label{ds}
Let $ \phi \in \mathfrak{C} $, $ T \in (0,+\infty] $ and $ u_0 \in L^\infty(M)$, with $ u_0 \ge 0 $. We say that a nonnegative function $ u \in L^\infty(M \times (0,T)) $ is a distributional solution to problem \eqref{gen-zero-bdd} if it satisfies
\begin{equation}\label{def-sol}
\int_0^T \int_M u \, \xi_t  \, dx dt + \int_0^T \int_M \phi(u) \, \Delta \xi \, dx dt + \int_M u_0(x) \, \xi(0,x) \, dx = 0
\end{equation}
for all $ \xi \in C^2_c( M \times [0,T) ) $. Similarly, we say that a nonnegative function $ u \in L^\infty(M \times (0,T) ) $ is a distributional supersolution [subsolution] to \eqref{gen-zero-bdd} if \eqref{def-sol} is satisfied with ``='' replaced by ``$\le$'' [``$ \ge $''], for all nonnegative $ \xi \in C^2_c( M \times [0,T) ) $.
\end{den}

% From here on we shall generally refer to ``solutions'' to \eqref{gen-zero-bdd},

\begin{pro}\label{minimal}
Let $ \phi \in \mathfrak{C} $, $ T = +\infty $ and $ u_0 \in L^\infty(M)$, with $ u_0 \ge 0 $. There exists a (minimal) solution to \eqref{gen-zero-bdd}, in the sense of Definition \ref{ds}, which in particular satisfies $ \| u(t) \|_\infty \le \| u_0 \|_\infty $ for a.e.\ $ t>0 $.
\end{pro}
\begin{proof}[Sketch of proof]
Let $ B_R $ be the ball of radius $ R>0 $ centered at any point $ o \in M $. One considers the approximate problems
\begin{equation*}\label{sol-approx-zero}
\begin{cases}
\left( u_R \right)_t = \Delta \phi\!\left( u_R \right)   & \text{in } B_R \times \mathbb{R}^+  \, , \\
u_R(\cdot,t) = 0 &  \text{on } \partial B_R \times  \mathbb{R}^+ \, , \\
u_R(\cdot,0) = u_0  & \text{in } B_R  \, ,
\end{cases}
\end{equation*}
which admit a unique (nonnegative) strong energy solution for each $ R>0 $, namely a bounded function $ u_R $ such that $ \nabla \phi(u_R) \in L^2(B_R \times (0,T)) $ and $ (u_R)_t \in L^1(B_R \times (0,T)) $. See \cite{V} and references therein for analogous results in the Euclidean setting, that can easily be extended to the present one. Such a solution is in fact the \emph{standard} one (see Appendix \ref{app}). In this class of solutions and sub/supersolutions the comparison principle holds, so that the sequence $ \{u_R\} $  is pointwise increasing with respect to $R$ and bounded above by $ \| u_0 \|_\infty $. Hence $ u:=\lim_{R\to\infty} u_R $ is indeed a solution to \eqref{gen-zero-bdd}, which is minimal again by comparison on balls. For all of the comparison results we have exploited we refer to Proposition \ref{CP} in the Appendix (when comparing $ u_{R} $ and $ u_{R+1} $ note that the fact that $ u_{R+1} $ is defined in $ B_{R+1} $ rather than the whole $M$ is irrelevant).
\end{proof}

%In the Euclidean framework, for the model case $\phi(u)=u^m$ with $m\in(0,1)$, existence holds for much more general initial data (see \cite{HP}), namely for $u_0\in L^1_{\mathrm{loc}}({\mathbb R}^N)$. This is of course not extendible to the present context since, for instance, the function $\phi(u)=u$ belongs to $\mathfrak{C}$.

%\smallskip

As concerns the elliptic equation \eqref{elliptic-phi}, we give the following standard definition.

\begin{den}\label{vw}
Let $ \phi \in \mathfrak{C} $. We say that a nonnegative function $ W \in L^\infty(M) $ is a distributional solution to equation \eqref{elliptic-phi} if it satisfies
\begin{equation*}\label{def-sol-ell}
 \int_M W \, \Delta \eta  \, dx = \int_M \phi^{-1}(W) \, \eta  \, dx
\end{equation*}
for all $ \eta \in C^2_c( M ) $. Similarly, we say that a nonnegative function $ W \in L^\infty(M ) $ is a distributional supersolution [subsolution] to \eqref{elliptic-phi} if \eqref{vw} is satisfied with ``='' replaced by ``$\le$'' [``$ \ge $''], for all nonnegative $ \eta \in C^2_c( M  ) $.
\end{den}

%\subsection{Geometric preliminaries} To be written (short)

%\subsection{On the differential equations considered} To be written (short)

\section{Proof of the main result}\label{proofmain}

%Note that $ (\ref{b}) \Rightarrow (\ref{c}) $ and $ (\ref{d}) \Rightarrow (\ref{e}) $ are trivial, while $ (\ref{b}) \Rightarrow (\ref{d}) $ will just follow from the proof of $ (\ref{c}) \Rightarrow (\ref{e}) $, since the function $ \phi $ we use in order to pass from $ (\ref{c}) $ to $ (\ref{e}) $ is the same. Therefore, below we shall only focus on the remaining implications.

\begin{proof}[Proof of $(\ref{a}) \Rightarrow (\ref{b})$]
First of all, it is convenient to rewrite the differential equation in \eqref{gen-zero-bdd} as follows:
\begin{equation}\label{eq:inv}
\psi^\prime(w) \, w_t = \Delta w \, .
\end{equation}
By virtue of the assumptions on $ \phi $, we know that $ \psi(w)=\phi^{-1}(w) $ is a \emph{convex} function having the same properties as $ \phi $, except that it is defined in the interval $ [0,a) $, with $ a:= \lim_{u \to +\infty} \phi(u) \in (0,+\infty] $. Thanks to \cite[Theorem 8.18]{GrigHK}, stochastic incompleteness ensures the existence of a nonnegative, nontrivial, bounded solution $ v $ to
\begin{equation}\label{v-SI}
\Delta v = \alpha v \qquad \text{in } M \, ,
\end{equation}
for every $ \alpha>0 $; we shall choose later a precise value of $ \alpha $. For simplicity, we can and shall assume that $ \| v \|_\infty = 1 $. The starting point is to look for a subsolution to \eqref{eq:inv} of the form
\begin{equation}\label{eq:ee}
w(x,t) = f(t) \, v(x) \, ,
\end{equation}
where $ f $ is a suitable positive, increasing function that we shall specify below. An immediate computation shows that $ w $ is actually a subsolution to \eqref{eq:inv} if and only if
\begin{equation*}\label{eq:inv-2}
\psi^\prime(f v) \, f_t \le \alpha f \, ,
\end{equation*}
which is trivially implied by
\begin{equation*}\label{eq:inv-3}
\psi^\prime(f) \, f_t \le \alpha f
\end{equation*}
given the monotonicity of $ \psi^\prime $ (recall that $ \psi $ is convex), the fact that $ f $ is increasing and the bound $ \| v \|_\infty = 1 $. So, let us pick $f$ as the solution to the Cauchy problem
\begin{equation}\label{choice-f}
\begin{cases}
f_t = \frac{\alpha f}{\psi^\prime(f)} \, , \\
f(0) = \varepsilon \in (0,a) \, ,
\end{cases}
\end{equation}
where $ \varepsilon$ for the moment is a free parameter.
%and $ [0,a) $ is the maximal interval on which $ \psi $ is defined. Note that $ a := \lim_{u \to + \infty} \phi(u) $ and either $ a \in (0,+\infty) $ or $ a=+\infty $. In any case, we have that $ \lim_{w \to a^-} \psi(w) = \lim_{w \to a^-} \psi^\prime(w)=+\infty $.
Clearly, such a function is increasing and, the r.h.s.~of the differential equation being sublinear, it exists for all times $ t \in \mathbb{R}^+ $ (at this stage one could also use the variable $ g=\psi(f) $ and study the differential equation $ g_t=\phi(g) $). Let us set
$$
F(x):= \frac{1}{\alpha} \int_\varepsilon^x \frac{\psi^\prime(y)}{y} \, dy \qquad \forall x \in [\varepsilon,+\infty) \, .
$$
By integrating \eqref{choice-f}, we deduce that
\begin{equation*}\label{choice-f-1}
F(f(t)) = t \qquad \forall t \in \mathbb{R}^+ \, .
\end{equation*}
Because $F$ is locally regular in $ (0,a) $ and  $ \lim_{x \to a^-} F(x) = +\infty $, we infer that necessarily
\begin{equation}\label{choice-f-2}
\lim_{t\to+\infty} f(t) = a \qquad \Longrightarrow \qquad \lim_{t\to+\infty} \psi\! \left( f(t) \right) = + \infty  \, .
\end{equation}
Our next aim is to exhibit a subsolution to \eqref{gen-zero-bdd} which, at a certain time (smaller than $T$), exceeds the $ L^\infty $ norm of the initial datum. To this end, given $ w $ as in \eqref{eq:ee} with the above choices, let us consider the function
$$
\underline{w}:= \left( w-\varepsilon \right) \vee 0 \, .
$$
Clearly, $ \underline{w}(0) = \left( \varepsilon v -\varepsilon \right) \vee 0 = 0 $ and $ \underline{w} $ is still a subsolution to the differential equation \eqref{eq:inv}, since
$$
\Delta \underline{w} = \Delta \left( \left( w-\varepsilon \right) \vee 0 \right) \ge \chi_{w \ge \varepsilon} \, \Delta w \ge  \chi_{w \ge \varepsilon} \, w_t \, \psi^\prime(w) = \underline{w}_t \, \psi^\prime(w) \ge \underline{w}_t \, \psi^\prime(\underline{w}) \, ,
$$
where we have used Kato's inequality along with the fact that $ w_t \ge 0 $ and $ \psi^\prime $ is a nondecreasing function. Hence, going back to the original problem, we have constructed the following subsolution to \eqref{gen-zero-bdd}:
$$
\underline{u}(x,t) := \psi(\underline{w}(x,t)) = \psi\!\left( \left( f(t)\,v(x)-\varepsilon \right) \vee 0 \right) ;
$$
in particular,
\begin{equation}\label{eq-infty}
\left\| \underline{u}(t) \right\|_\infty = \psi\!\left( \left( f(t)-\varepsilon \right) \vee 0 \right) .
\end{equation}
Since $ \lim_{w \to a^-} \psi(w) = +\infty $, we can take $ \varepsilon $ so small that
\begin{equation*}\label{eq-infty-bis}
\psi\!\left( b - 2\varepsilon  \right) > \left\| u_0 \right\|_\infty ,
\end{equation*}
where $ b \in (0,a) $ is any number such that $ \psi\!\left( b \right) > \left\| u_0 \right\|_\infty $. By \eqref{choice-f-2} and \eqref{eq-infty}, we deduce that
\begin{equation}\label{eq-infty-ter}
\left\| \underline{u}(S) \right\|_\infty > \left\| u_0 \right\|_\infty
\end{equation}
provided $ S>0 $ is so large that $ f(S) \ge b-\varepsilon $. In order to make sure that $ S < T $, we can exploit a simple scaling argument: just note that $ f(t)=f_0(\alpha t) $, the function $f_0$ being the solution to \eqref{choice-f} corresponding to $ \alpha=1 $. Hence, if $ S_0 $ is the time at which $ f_0(S_0) = b-\varepsilon $, it is enough to take $ \alpha $ so large that $ S_0 / \alpha < T $.\\
Finally, we are able to construct a \emph{solution} to \eqref{gen-zero-bdd} which stays above $ \underline{u} $ in $ M \times (0,S) $, by means of the following approximate Cauchy-Dirichlet problems (let $ R>0 $):
\begin{equation}\label{sol-approx-dir-bb}
\begin{cases}
\left( u_R \right)_t = \Delta \phi\!\left( u_R \right)   & \text{in } B_R \times \mathbb{R}^+  \, , \\
u_R(\cdot,t) = \psi\!\left (f(t) \vee \phi(\| u_0 \|_\infty)\right) &  \text{on } \partial B_R \times  \mathbb{R}^+ \, , \\
u_R(\cdot,0) = u_0  & \text{in } B_R  \, .
\end{cases}
\end{equation}
It is apparent that $ \underline{u} $ is a subsolution to \eqref{sol-approx-dir-bb} for all $R>0$. As a consequence of the comparison principle on balls (we refer to Proposition \ref{CP}), along with the fact that the boundary condition is given by an increasing function of $t$, we can easily infer that
\begin{equation*}\label{std-comp-bb}
\underline{u} \le u_{R+1} \le u_{R} \le \psi\!\left(f(t) \vee \phi(\| u_0 \|_\infty)\right) \qquad \text{in } B_R \times \mathbb{R}^+
\end{equation*}
for all $ R>0 $. In particular, $ \{ u_R \} $ (e.g.~set to zero outside $B_R$) is a monotone-decreasing sequence of solutions to \eqref{sol-approx-dir-bb} which stays bounded in $ L^\infty(M \times (0,\tau)) $ for all $ \tau>0 $. Hence, by passing to the limit as $ R \to \infty $, we find that $ u_1 := \lim_{R \to \infty} u_R $ is a solution to \eqref{gen-zero-bdd} satisfying the additional estimate
\begin{equation}\label{std-comp-1-bb}
\underline{u} \le u_1 \quad \text{in } M \times \mathbb{R}^+ \qquad \Longrightarrow \qquad \left\| u_0 \right\|_\infty < \left\| u_1(S) \right\|_\infty ,
\end{equation}
where we have exploited \eqref{eq-infty-ter}. On the other hand, one can always construct the \emph{minimal} solution to \eqref{gen-zero-bdd}, namely the one obtained by the same approximation scheme as above, upon replacing the boundary condition in \eqref{sol-approx-dir-bb} with $ 0 $ on $ \partial B_R \times \mathbb{R}^+ $ (see Proposition \ref{minimal}). By proceeding in this way, we end up with another solution $ u_2 $ to \eqref{gen-zero-bdd} (also existing in $ M \times \mathbb{R}^+ $), which in particular satisfies
\begin{equation}\label{std-comp-2-bb}
u_2 \le \left\| u_0 \right\|_{\infty} \qquad \text{in } M \times \mathbb{R}^+ \, .
\end{equation}
Because \eqref{std-comp-1-bb} and \eqref{std-comp-2-bb} are clearly incompatible and $ S < T $, we necessarily deduce that $ u_1 \not \equiv u_2 $ in $ M \times (0,T) $. Actually we have to point out that, by construction, $ u_1 $ is in $ L^\infty(M \times (0,\tau)) $ for all $ \tau>0 $ but does not belong to $ L^\infty(M \times \mathbb{R}^+) $. In order to make $ u_1 $ globally bounded (which is required in the case $ T=+\infty $ only), but still different from $ u_2 $, it is enough to stop it at time $S$ and then restart it e.g.~with a minimal construction.
\end{proof}

%\TR{Qui forse sarebbe interessante dedurre la non-unicita' in ogni intervallo $ (0,T) $, per ogni $T>0$. Per ora questa non-unicita' piu' forte ce l'avremmo solo per il dato $0$. Questo pero' credo si possa semplicemente fare usando la soluzione, normalizzata a uno, di $ \Delta v = \alpha v $ con $\alpha$ suff grande}.
% Ma comunque non so se sia davvero interessante e/o fattibile, nel senso che in ogni caso sotto ipotesi di completezza stocastica abbiamo effettivamente unicita' in ogni intervallo.

\begin{proof}[Proof of $(\ref{b}) \Rightarrow (\ref{c})$]
Since the \emph{minimal} solution to \eqref{gen-zero-bdd} always exists, which we shall simply denote by $ u $, from assumption $(\ref{b})$ we deduce the existence of another solution $ u_\ast \in L^\infty(M \times (0,T)) $, which is necessarily larger than $ u $ (by minimality) but different from $ u $. By multiplying both $ u_\ast $ and $u$ by $ e^{-t} $ and subtracting, we obtain the following identity:
\begin{equation}\label{eq:ww1}
\left[ e^{-t} \left( u_\ast - u \right) \right]_t = e^{-t} \left[ \Delta \phi(u_\ast) - \Delta \phi(u) \right] - e^{-t} \left( u_\ast - u \right) .
\end{equation}
Now let us integrate \eqref{eq:ww1} between $ t=0 $ and $ t = T $: by exploiting the fact that $ u_\ast \ge u $ and $ u_\ast(\cdot,0) = u(\cdot,0) $, we end up with the inequality
\begin{equation}\label{eq:g1}
\Delta \left[ \int_0^{T} \phi(u_\ast(x,s)) \, \hat{e}(s) \, ds - \int_0^{T} \phi(u(x,s)) \, \hat{e}(s) \, ds \right] \ge \int_0^{T} \left[ u_\ast(x,s) - u(x,s) \right] \hat{e}(s) \, ds \, ,
\end{equation}
valid in the whole $M$, where we set $ \hat{e}(t) := e^{-t} / \left(1-e^{-T}\right) $ if $ T<+\infty $ and $ \hat{e}(t)=e^{-t} $ if $ T=+\infty $. Thanks to the concavity of $ \phi $, Jensen's inequality ensures that
$$
\int_0^{T} \phi\!\left( u_\ast(x,s) - u(x,s) \right) \hat{e}(s) \, ds \le \phi \! \left( \int_0^{T} \left[ u_\ast(x,s) - u(x,s) \right] \hat{e}(s) \, ds \right) ,
$$
so that by \eqref{eq:g1} and the monotonicity of $ \phi $ we infer that
\begin{equation}\label{eq:g2}
\Delta \left[ \int_0^{T} \phi(u_\ast(x,s)) \, \hat{e}(s) \, ds - \int_0^{T} \phi(u(x,s)) \, \hat{e}(s)  \, ds \right] \! \ge \phi^{-1} \! \left( \int_0^{T} \phi\!\left( u_\ast(x,s) - u(x,s) \right) \hat{e}(s) \, ds  \right) .
\end{equation}
Still the concavity of $ \phi $ (in the form of decrease of difference quotients), along with the fact that $ \phi(0)=0 $, yields
\begin{equation}\label{eq:g2-bis}
\phi\!\left( u_\ast(x,s) - u(x,s) \right) \ge \phi(u_\ast(x,s)) - \phi(u(x,s)) \, .
\end{equation}
Hence, from \eqref{eq:g2}, \eqref{eq:g2-bis} and again the monotonicity of $ \phi $, we have:
\begin{equation*}\label{eq:g3}
\begin{aligned}
& \Delta \left[ \int_0^{T} \phi(u_\ast(x,s)) \, \hat{e}(s) \, ds - \int_0^{T} \phi(u(x,s)) \, \hat{e}(s)  \, ds \right] \\
\ge & \, \phi^{-1} \! \left( \int_0^{T} \phi(u_\ast(x,s)) \, \hat{e}(s) \, ds - \int_0^{T} \phi(u(x,s)) \, \hat{e}(s) \, ds \right) .
\end{aligned}
\end{equation*}
Summing up, we have proved that the function
$$
\underline{W}(x) := \int_0^{T} \phi(u_\ast(x,s)) \, \hat{e}(s) \, ds - \int_0^{T} \phi(u(x,s)) \, \hat{e}(s) \, ds \, ,
$$
which is nonnegative, nontrivial and bounded by construction, satisfies
$$
\Delta \underline{W} \ge \phi^{-1}\!\left( \underline{W} \right)  \qquad \text{in } M \, ,
$$
namely it is a \emph{subsolution} to \eqref{elliptic-phi}. On the other hand, it is plain that the constant function $ \overline{W} \equiv \| \underline{W} \|_\infty $ is a \emph{supersolution} to the same equation. We can therefore construct the claimed solution $ W $ by the following standard approximation scheme. For all $ R > 0 $, we solve the Dirichlet problems
\begin{equation}\label{pb-balls-kk}
\begin{cases}
\Delta W_R = \phi^{-1} \! \left( W_R \right)  & \text{in } B_R \, , \\
W_R = \overline{W}  &  \text{on } \partial B_R \, . \\
\end{cases}
\end{equation}
Clearly, for each $ R>0 $, $ \underline{W} $ and $ \overline{W} $ are two ordered sub- and supersolutions, respectively, to \eqref{pb-balls-kk}. By comparison in $B_R$ (see Proposition \ref{CE}), we deduce that
\begin{equation*}\label{eq:est-R-kk}
\underline{W} \le W_{R+1} \le W_R \le \overline{W} \qquad \text{in } B_R \, .
\end{equation*}
Hence, upon letting $ R \to \infty $, we find that $ W:=\lim_{R\to\infty} W_R $ is indeed a nonnegative solution to \eqref{elliptic-phi}, which is nontrivial and bounded since $ W \ge \underline{W} $ and $ W \le \overline{W} $ in $M$.
\end{proof}

\begin{proof}[Proof of $(\ref{c}) \Rightarrow (\ref{a})$]
As a preliminary step, we observe that \eqref{elliptic-phi} yields
\begin{equation}\label{eq-psi-1}
\Delta W \ge \frac{W}{2} \, \psi^\prime\!\left( \frac W 2 \right)  \qquad \text{in } M \, ,
\end{equation}
where $ \psi(w):=\phi^{-1}(w) $ is the same convex function introduced in the proof of implication $(\ref{a}) \Rightarrow (\ref{b})$. Note that in \eqref{eq-psi-1} we have just used the fact that $ \psi(w) $ lies above the tangent line at $ \frac{w}{2} $, along with the positivity of $ \psi(\frac{w}{2}) $. So from \eqref{eq-psi-1} we can deduce the existence of a nonnegative, nontrivial, bounded function satisfying
\begin{equation*}\label{eq-psi-bis}
\Delta W \ge \frac{W}{2} \, \psi^\prime\!\left( W \right)  \qquad \text{in } M \, ,
\end{equation*}
which for simplicity we shall keep denoting by $W$.
%In order to prove it, we can follow exactly the same reasoning as in the end of the proof of implication $(\ref{b}) \Rightarrow (\ref{c})$, just upon replacing the function $ \phi^{-1}(w) $ with $ \frac{w}{2} \, \psi^\prime(w) $.\\
We look for a (local-in-time) subsolution to \eqref{eq:inv} of the form
$$
\underline{w}(x,t) = f(t) \, W(x)
$$
for a suitable positive, increasing function $f$. An elementary computation shows that, to this aim, we need to require that
\begin{equation*}\label{eq-subsol-gen}
\psi^\prime\!\left( f \, W \right) f_t \, W \le f \, \frac{W}{2} \, \psi^\prime(W) \, ,
\end{equation*}
which is trivially implied by
$$
f_t \le \frac{f}{2}
$$
as long as $ f \le 1 $, recalling the fact that $ \psi^\prime $ is nondecreasing. As a consequence, we deduce that (for instance) the function
$$
\underline{w}(x,t) := \frac{e^{\frac{t}{2}}}{2} \, W(x)
$$
is indeed a subsolution to \eqref{eq:inv} at least up to $ t = \log 4 $, namely
$$
\underline{u}(x,t) = \psi\!\left( \frac{e^{\frac{t}{2}}}{2} \, W(x) \right) \qquad \forall t \in [0,\log 4]
$$
is a subsolution to $ u_t = \Delta \phi(u) $, going back to the original variables. Note that
\begin{equation}\label{eq:evo-linf}
\left\| \underline{u}(t) \right\|_\infty =  \psi\!\left( \frac{e^{\frac{t}{2}}}{2} \, \| W \|_\infty \right) .
\end{equation}
We are therefore in position to construct a bounded solution to the Cauchy problem
\begin{equation}\label{k1}
\begin{cases}
u_t = \Delta \phi(u) & \text{in } M \times \mathbb{R}^+ \, , \\
u(\cdot,0)= \left\| \underline{u}(0) \right\|_\infty = \psi\!\left( \frac{1}{2} \, \| W \|_\infty \right) =: c > 0 & \text{in } M  \, ,
\end{cases}
\end{equation}
which is \emph{larger} than the positive constant $ c $. Indeed, it is enough to solve the approximate problems (for all $R>0$)
\begin{equation}\label{k2}
\begin{cases}
\left( u_R \right)_t = \Delta \phi\!\left( u_R \right)   & \text{in } B_R \times \mathbb{R}^+  \, , \\
u_R(\cdot,t) =  \psi\!\left( \frac{e^{\frac{t}{2}}}{2} \, \| W \|_\infty \right) \wedge \psi\!\left( \| W \|_\infty \right) \ge c &  \text{on } \partial B_R \times  \mathbb{R}^+ \, , \\
u_R(\cdot,0) = c >0 & \text{in } B_R  \, ,
\end{cases}
\end{equation}
and let $ R \to \infty $. Since $ \underline{u} $ is a subsolution to \eqref{k2}, for each $ R>0 $, up to $ t=\log 4 $, by arguing similarly to the proof of implication $ (\ref{a}) \Rightarrow (\ref{b}) $ we deduce that $ u_\ast := \lim_{R \to \infty} u_R $ (monotone decreasing limit w.r.t.~$R$) is a solution to \eqref{k1} satisfying
$$
c \le u_\ast \le \psi\!\left( \| W \|_\infty \right) \quad \text{in } M \times \mathbb{R}^+ \qquad \text{and} \qquad u_\ast \ge \underline{u} \quad \text{in } M \times [0,\log 4] \, ;
$$
in particular, recalling \eqref{eq:evo-linf}, there holds
\begin{equation*}\label{kkk}
\left\| u_\ast(t) \right\|_\infty \ge \psi\!\left( \frac{e^{\frac{t}{2}}}{2} \, \| W \|_\infty \right) \qquad \forall t \in [0,\log 4] \, ,
\end{equation*}
so that $ u_\ast $ is a bounded solution to \eqref{k1}, everywhere not smaller than $c$ since such a constant is trivially a subsolution to \eqref{k2} for all $R>0$ (recall again Proposition \ref{CP}), but not identically $c$. By performing the same computations as in the first part of the proof of implication $(\ref{b}) \Rightarrow (\ref{c})$, using $ u \equiv c $ and $ T=+\infty $ there, we obtain the following identity:
\begin{equation}\label{eq:g1-ast}
\Delta \left[ \int_0^{+\infty} \phi(u_\ast(x,s)) \, e^{-s}  \, ds \right] = \int_0^{+\infty} \left[ u_\ast(x,s) - c \right] e^{-s}  \, ds \, .
\end{equation}
Now note that, as $ \phi $ is concave,
\begin{equation}\label{eq:g1-conc}
\phi(u_\ast) \le \phi(c) + \phi^\prime(c) \left( u_\ast - c \right) ,
\end{equation}
where $ \phi^\prime(c)>0 $ because $ \phi $ is strictly monotone. Hence, by combining \eqref{eq:g1-ast} and \eqref{eq:g1-conc}, we end up with the inequality
\begin{equation*}\label{eq:g2-ast}
\Delta \left[ \int_0^{+\infty} \phi(u_\ast(x,s)) \, e^{-s}  \, ds \right] \ge \frac{1}{\phi^\prime(c)} \int_0^{+\infty} \left[ \phi\!\left( u_\ast(x,s) \right) - \phi(c) \right] e^{-s}  \, ds \, ,
\end{equation*}
namely the function
$$
V(x) := \int_0^{+\infty} \left[ \phi\!\left( u_\ast(x,s) \right) - \phi(c) \right] e^{-s}  \, ds
$$
is nonnegative, nontrivial, bounded and satisfies
\begin{equation*}\label{eq:fin}
\Delta V \ge \frac{1}{\phi^\prime(c)} \, V  \qquad \text{in } M \, .
\end{equation*}
Starting from $V$, we can then construct a nonnegative, nontrivial, bounded solution to
$$ \Delta v = \frac{1}{\phi^\prime(c)} \, v \qquad \text{in } M $$
just by performing again the same approximation scheme as in the end of the proof of implication $(\ref{b}) \Rightarrow (\ref{c})$. This ensures that the manifold $M$ is stochastically incomplete, due again to \cite[Theorem 8.18]{GrigHK}.
\end{proof}

\begin{oss} \rm In the proof of $(\ref{a}) \Rightarrow (\ref{b})$ we exploited the assumption of stochastic incompleteness only in terms of existence of nontrivial solutions to \eqref{v-SI} for \it any $\alpha>0$. \rm On the other hand, in the proof of $(\ref{c}) \Rightarrow (\ref{a})$ we established stochastic incompleteness in terms of existence of nontrivial solutions to \eqref{v-SI} for a \emph{specific} $\alpha_0>0$. Nonetheless, it is easy to show directly that this yields existence for \it any $\alpha>0$. \rm Indeed, let $\widetilde v$ be a nonnegative, nontrivial, bounded solution to $\Delta \widetilde v=\alpha_0\widetilde v$. Clearly $\widetilde v$ is a subsolution to $\Delta v=\alpha v$ for every $\alpha\in(0,\alpha_0)$. Moreover, for all $\beta>1$ Kato's inequality ensures that $ v_\beta := \widetilde v^\beta$ is a subsolution to $\Delta v=\beta\alpha_0 v$. Hence our method of proof shows that, under assumption $(\ref{c})$, there is a nonnegative, nontrivial, bounded subsolution to $\Delta v=\alpha v$ for any $\alpha>0$. By exploiting $ v_\beta $ and proceeding exactly as in the end of the proof of  $(\ref{b}) \Rightarrow (\ref{c})$, one can then show the existence of a nonnegative, nontrivial, bounded \emph{solution} to $\Delta v=\alpha v$ for any $\alpha>0$. Going back to our proof of
$(\ref{a}) \Rightarrow (\ref{b})$, in this way we are able to provide a \emph{direct} proof of the implication $(\ref{c}) \Rightarrow (\ref{b})$, thus establishing the equivalence between $ (\ref{b}) $ and $ (\ref{c}) $ independently of $ (\ref{a}) $. The equivalence between $(\ref{c})$ and $(\ref{a})$ can then be deduced as explained in Remark \ref{rem-ell}.
\end{oss}

%\begin{oss}\rm
% In the special but particularly relevant case $\phi(u)=u^m$, with $m\in(0,1)$ (namely the \emph{fast diffusion equation}), a simple proof of the implication $ (\ref{c}) \Rightarrow (\ref{b})$ could be provided. Indeed, given a positive, bounded, nontrivial solution $V$ of the equation $\Delta V=V^{\frac1m}$, one constructs a separable solution $U$ to the parabolic equation $u_t=\Delta u^m$ of the form $U(x,t)=c_m\,t^{\frac1{1-m}}\,V^{\frac1m}$ for a suitable $c_m>0$. This is clearly a solution starting from the initial datum $u_0\equiv0$, so that nonuniqueness is shown for that initial datum. Besides, $U$ is a subsolution to the problem $u_t=\Delta u^m$, $u(\cdot, 0)=u_0\ge0$, and from this it is standard to construct a solution to such problem that differs from the \it minimal \rm one, since its $L^\infty$ norm is by construction larger.
%\end{oss}

\appendix

\section{Comparison principles on balls}\label{app}

The aim of this appendix is to justify the comparison principles of which we take advantage in the proof of Theorem \ref{main-teo}, both for parabolic and elliptic problems. We only need to compare suitable solutions with general distributional sub/supersolutions. We do not intend to provide comparisons under the most general assumptions, but only to our specific purposes.

\smallskip

Given $ \phi \in \mathfrak{C} $, $ R>0 $, $ T>0 $, $ u_0 \in L^\infty(M) $ with $ u_0 \ge 0 $ and $ g \in L^\infty_{\rm loc}([0,+\infty)) $ with $ g \ge 0 $, we say that a nonnegative function $ u_R \in L^\infty(M \times (0,T)) $ is the \emph{standard solution} to the Cauchy-Dirichlet problem
\begin{equation}\label{std-prob-R}
\begin{cases}
\left( u_R \right)_t = \Delta \phi\!\left( u_R \right)   & \text{in } B_R \times (0,T)  \, , \\
u_R(\cdot , t) = g(t) &  \text{on } \partial B_R \times (0,T) \, , \\
u_R(\cdot,0) = u_0  & \text{in }  B_R  \, ,
\end{cases}
\end{equation}
if it is obtained as a monotone limit, as $ \varepsilon \downarrow 0 $, of the solutions $ u_R^\varepsilon $ (let $ \varepsilon>0 $) to the approximate (quasilinear) problems
\begin{equation}\label{std-prob-R-eps}
\begin{cases}
\left( u_R^\varepsilon \right)_t = \Delta \phi\!\left( u_R^\varepsilon \right)   & \text{in } B_R \times (0,T)  \, , \\
u_R^\varepsilon(\cdot,t) = g(t) + \varepsilon &  \text{on } \partial B_R \times (0,T) \, , \\
u_R^\varepsilon(\cdot,0) = u_0 +\varepsilon & \text{in } B_R  \, .
\end{cases}
\end{equation}

\smallskip

For simplicity, and in order not to weigh down the discussion, we shall suppose that $M$ is geodesically complete and that each $ B_R $ has a smooth boundary. This is of course not necessarily true, since on general manifolds $  B_R$ is only Lipschitz regular. More rigorously, instead of working in $ B_R $, one should consider appropriate sublevel sets of a smooth \emph{exhaustion function} of $M$,  which can always be constructed independently of geodesic completeness (see e.g.~\cite[Proposition 2.28, Proposition 5.47 and Theorem 6.10]{Lee}).

\begin{pro}[Comparison for parabolic problems]\label{CP}
Let $ u_R $ be the standard solution of \eqref{std-prob-R} and $ \overline{u} $ [$ \underline{u} $] be a distributional supersolution [subsolution] to \eqref{gen-zero-bdd}, in the sense of Definition \ref{ds}. Suppose that $ \overline{u} \ge g(t) $ [$ \underline{u} \le g(t) $] a.e.~in $ B_R \times (0,T) $. Then $ \overline{u} \ge u_R $ [$ \underline{u} \le u_R $] a.e.~in $ B_R \times (0,T) $.
\end{pro}
\begin{proof}
The argument is a dual one, to some extent classical, so we shall be concise and stress the most critical points. Our main ideas are borrowed from \cite{PZ} and \cite{GMP}. We shall only prove the result for supersolutions, since the proof for subsolutions is completely analogous.
% In the sequel we make the additional assumption that $u_R$ and $ \overline{u} $ \TR{(and therefore $g$)} are continuous: this is just to avoid further technicalities, otherwise one should exploit the fact that for \emph{almost every} radius $ R>0 $ such functions have a well-defined trace on $ \partial B_R $, in the sense of Lebesgue points.
First of all, a time cut-off argument immediately shows that the inequality
\begin{equation}\label{n3}
\begin{aligned}
& \int_0^S \int_{B_R} \left[ (\overline{u} - u_R^\varepsilon) \, \xi_t + \left( \phi( \overline{u} ) - \phi(u_R^\varepsilon) \right) \Delta \xi \right] dx dt  \\
\le & \int_{B_R} \left[ \overline{u}(x,S) - u_R^\varepsilon(x,S) \right] \xi(x,S) \, dx + \varepsilon \int_{B_R} \xi(x,0) \, dx
\end{aligned}
\end{equation}
holds for almost every $ S \in (0,T) $ and any nonnegative $ \xi \in C^\infty_c(B_R \times [0,S]) $. Let us introduce the following function, defined in $ B_R \times (0,T) $:
\begin{equation*}\label{n4}
a_\varepsilon(x,t):=
\begin{cases}
\frac{\phi(\overline{u}(x,t)) - \phi(u_R^\varepsilon(x,t))}{\overline{u}(x,t) - u_R^\varepsilon(x,t)} & \text{if } \overline u (x,t) \neq u_R^\varepsilon(x,t) \, , \\
0 & \text{if } \overline u(x,t) = u_R^\varepsilon (x,t) \, .
\end{cases}
\end{equation*}
Note that, because $ \phi $ is increasing and $ C^1(\mathbb{R}^+ \setminus \{ 0 \}) $, the functions $ \overline{u} , u_R^\varepsilon $ are bounded and by construction  $ u_R^\varepsilon \ge \varepsilon $, in fact $ a_\varepsilon $ is nonnegative and bounded in $ B_R \times (0,T) $. Given an arbitrary nonnegative function $ \omega \in C^\infty_c(B_R) $, we can take $ \xi \equiv \xi_\varepsilon $ as the (classical) solution to the following backward parabolic problem:
\begin{equation}\label{n9}
\begin{cases}
(\xi_\varepsilon)_t  + a_\varepsilon(x,t) \, \Delta \xi_\varepsilon = 0 & \text{in } B_R \times (0, S) \, , \\
\xi_\varepsilon(\cdot,t) = 0 & \text{on } \partial B_R \times (0,S) \, , \\
\xi_\varepsilon(\cdot,S) = \omega & \text{in } B_R \, .
\end{cases}
\end{equation}
Actually, here there are two issues: $ a_\varepsilon $ is not a smooth and positive coefficient and $ \xi_\varepsilon $ is not compactly supported inside $ B_R $. The first one can be handled by a routine approximation with regular and bounded away from zero coefficients.  More precisely, following \cite[Proof of Theorem 2.3]{GMP}, since $ a_\varepsilon $ is nonnegative and bounded by a local regularization and lifting procedure we can construct a smooth and positive function $ \tilde{a}_\varepsilon $ such that
\begin{equation}\label{n48-bis}
\int_0^T \int_{B_R} \frac{\left(a_\varepsilon -\tilde{a}_\varepsilon\right)^2}{\tilde{a}_\varepsilon} \, dx dt \leq \varepsilon \, .
\end{equation}
Then, we solve \eqref{n9} with $ \tilde{a}_\varepsilon $ replacing $ a_\varepsilon $ (for simplicity we still denote by $ \xi_\varepsilon $ the corresponding solution). By linear parabolic theory, we can now assert that $ \xi_\varepsilon $ is indeed nonnegative, smooth and bounded by $  \| \omega \|_\infty $; if it were also spatially compactly supported in $B_R$, taking advantage of \eqref{n3} and the definition of $ a_\varepsilon $ we would end up with
\begin{equation}\label{n3bis}
\begin{aligned}
& \int_0^S \int_{B_R} \left(\overline{u} - u_R^\varepsilon\right) \left( a_ \varepsilon - \tilde{a}_\varepsilon \right) \Delta \xi_\varepsilon \, dx dt \le \int_{B_R} \left[ \overline{u}(x,S) - u_R^\varepsilon(x,S) \right] \omega(x) \, dx + \varepsilon \left\| \omega \right\|_\infty V_R \, ,
\end{aligned}
\end{equation}
where $ V_R $ stands for the volume of $ B_R $. On the other hand, the second issue we have pointed out above is actually related to the fact that $ \xi_\varepsilon $ is not spatially compactly  supported. Nevertheless, we will now prove  that \eqref{n3bis} still holds up to adding to the l.h.s.~the correct boundary terms involving the outer normal derivative $ \tfrac{\partial \xi_\varepsilon}{\partial \nu} $ of $ \xi_\varepsilon $ (which is nonpositive), that is
$$
\int_0^S \int_{\partial B_R} \left[ \phi(\overline{u}(x,t)) - \phi(u_R^\varepsilon(x,t)) \right] \left| \frac{\partial \xi_\varepsilon}{\partial \nu} (\sigma,t) \right| d\sigma dt \, ,
$$
where $ d\sigma $ is $ (N-1) $-dimensional Hausdorff measure on $ \partial B_R $. Since $ u_R^\varepsilon = g+\varepsilon $ and $ \overline{u} \ge g $ on $ \partial B_R $, we infer that \eqref{n3bis} can be replaced by
\begin{equation}\label{n3ter}
\begin{aligned}
& \int_0^S \int_{\partial B_R} \left[ \phi(g(t)) - \phi(g(t)+\varepsilon) \right] \left| \frac{\partial \xi_\varepsilon}{\partial \nu} (\sigma,t) \right| d\sigma dt +  \int_0^S \int_{B_R} \left(\overline{u} - u_R^\varepsilon\right) \left( a_ \varepsilon - \tilde{a}_\varepsilon \right) \Delta \xi_\varepsilon \, dx dt  \\
\le & \int_{B_R} \left[ \overline{u}(x,S) - u_R^\varepsilon(x,S) \right] \omega(x) \, dx + \varepsilon \left\| \omega \right\|_\infty V_R \, .
\end{aligned}
\end{equation}
In order to show this, we exploit a cut-off technique whose main ideas can be found in \cite[Proof of Theorem 2.3]{GMP}, although additional issues occur here. We select a family of cut-off functions $ \{ \eta_\delta \}_{0<\delta<R/2} \subset C^\infty_c(B_R) $ enjoying the following properties:
\begin{equation}\label{ciao}
\begin{gathered}
\eta_\delta(x) \equiv \eta_\delta(\rho(x)) \ \text{is radially nonincreasing w.r.t.~} \rho(x):=\operatorname{d}(x,o)\, ,  \\
 0 \le \eta_\delta \le 1 \, ,\qquad \eta_\delta \equiv 1 \quad \text{in } B_{R-2\delta} \, , \qquad \eta_\delta \equiv 0 \quad \text{in } B_{R-\delta}^c \, , \\
 \left| \nabla \eta_\delta \right| \le \frac{c}{\delta} \, , \qquad \left| \Delta \eta_\delta \right| \le \frac{c}{\delta^2} \, , \\
\lim_{\delta \downarrow 0} \int_{R-2\delta}^{R-\delta} f'(\rho) \, \eta_\delta'(\rho) \, d\rho = - f'(R) \qquad \text{and} \qquad \lim_{\delta \downarrow 0} \int_{R-2\delta}^{R-\delta} f(\rho) \, \eta_\delta''(\rho) \, d\rho = f'(R) \, ,
\end{gathered}
\end{equation}
where $ c>0 $ is a suitable constant independent of $ \delta $. Note that, whenever no ambiguity occurs, we identify $ \eta_\delta $ with the $1$-dimensional function of which it is the radial realization. Now, with no loss of generality, we can assume that $R$ is a \emph{Lebesgue radius} for $ \overline{u} $, in the sense that
\begin{equation}\label{leb-rad}
\lim_{\delta \downarrow 0} \frac{1}{\delta} \int_0^T \int_{B_{R-\delta} \setminus B_{R-2\delta}} \left| \phi(\overline{u}) - \overline{w}_\sigma \right| dxdt = 0
\end{equation}
where the bounded function $ \overline{w}_\sigma $ is the restriction of $\phi(\overline{u})$ to $\partial B_R$.
%, depending only on $t$ and the spherical coordinate $ \sigma \in \mathbb{S}^{N-1} $.
The same property also holds for $ \phi(u_R^\varepsilon) $ since, as a constructed solution, it belongs to the appropriate Sobolev spaces, hence it has a well-defined bounded trace on $ \partial B_R $ that coincides with $ \phi(g(t)+\varepsilon) $. Upon applying \eqref{n3} to $ \xi \equiv \eta_\delta \xi_\varepsilon $ and using the fact that $ \nabla\eta_\delta $ is supported in $ B_{R-\delta} \setminus B_{R-2\delta} $, we obtain:
\begin{equation}\label{int-pp}
\begin{aligned}
& \int_0^S \int_{B_R}  \left( \phi( \overline{u} ) - \phi(u_R^\varepsilon) \right) \Delta \!\left( \eta_\delta \xi_\varepsilon \right) dx dt \\
= & \int_0^S \int_{B_{R-\delta} \setminus B_{R-2\delta}}  \left( \phi( \overline{u} ) - \phi(u_R^\varepsilon) \right) \left( 2 \left\langle \nabla \xi_\varepsilon , \nabla \eta_\delta \right\rangle + \xi_\varepsilon \Delta \eta_\delta \right) dx dt \, .
\end{aligned}
\end{equation}
Note that we focus on this term only since all the other integrals in \eqref{n3} plainly converge as $ \delta \downarrow  0$. By virtue of the third property in \eqref{ciao}, \eqref{leb-rad} and the fact that $ \xi_\varepsilon $ vanishes on $ \partial B_R $, we infer:
\begin{equation}\label{int-pp-2}
\lim_{\delta \downarrow 0} \int_0^S \int_{B_{R-\delta} \setminus B_{R-2\delta}}  \left| \phi( \overline{u} ) - \phi(u_R^\varepsilon) -  \overline{w}_\sigma + \phi(g(t)+\varepsilon) \right| \left| 2 \left\langle  \nabla \xi_\varepsilon , \nabla \eta_\delta \right\rangle + \xi_\varepsilon \Delta \eta_\delta \right| dx dt = 0 \, .
\end{equation}
On the other hand, the fourth property in \eqref{ciao} (plus dominated convergence) and the expansion $ \xi_\varepsilon(x) \Delta \eta_\delta(x) = \xi_\varepsilon(x) \eta_\delta''(\rho(x)) + O(1) $ in $ B_{R-\delta} \setminus B_{R-2\delta} $ yield
\begin{equation}\label{int-pp-3}
\begin{aligned}
& \lim_{\delta \downarrow 0} \int_0^S \int_{B_{R-\delta} \setminus B_{R-2\delta}}  \left( w_\sigma - \phi(g(t)+\varepsilon) \right) \left( 2 \left\langle \nabla \xi_\varepsilon , \nabla \eta_\delta \right\rangle + \xi_\varepsilon \Delta \eta_\delta \right) dx dt \\
 = & \int_0^S \int_{\partial B_R}  -\left( w_\sigma - \phi(g(t)+\varepsilon) \right) \frac{\partial \xi_\varepsilon}{\partial \nu} (\sigma,t) \, d\sigma dt  \, ,
\end{aligned}
\end{equation}
recalling that $ \eta_\delta' $ is radial and the outer normal derivative coincides with the radial derivative. Since $ w_\sigma \ge \phi(g(t)) $ (consequence of the assumption from below on $ \overline{u} $) and $ \tfrac{\partial \xi_\varepsilon}{\partial \nu} \le 0 $, by letting $ \delta \downarrow 0 $ in \eqref{int-pp}, \eqref{int-pp-2} and \eqref{int-pp-3} we end up with \eqref{n3ter}.

By using the Green function of $ -\Delta $ in $B_R$ with homogeneous Dirichlet boundary conditions (multiplied by a suitable constant) as an upper barrier to \eqref{n9}, it is not difficult to deduce the estimate
\begin{equation}\label{est-1}
\left| \frac{\partial \xi_\varepsilon}{\partial \nu} \right| \le C \, ,
\end{equation}
where $ C $ is a positive constant depending only on $\| \omega \|_\infty $ and the aforementioned Green function in $B_R$ (in particular independent of $ \varepsilon $). Furthermore, if we multiply the differential equation in \eqref{n9} by $ \Delta \xi_\varepsilon $ and integrate by parts, we easily obtain the identity
\begin{equation}\label{est-2}
\frac12 \int_{B_R} \left| \nabla \xi_\varepsilon (x,S) \right|^2 dx + \int_0^S \int_{B_R} \tilde{a}_\varepsilon \left| \Delta \xi_\varepsilon \right|^2 dx dt = \frac12 \int_{B_R} \left| \nabla \omega \right|^2 dx  =: C' \, .
\end{equation}
Upon collecting \eqref{n48-bis}, \eqref{n3ter}, \eqref{est-1} and \eqref{est-2}, we end up with
$$
\begin{aligned}
& C \int_0^S \int_{\partial B_R} \left[ \phi(g(t)) - \phi(g(t)+\varepsilon) \right] d\sigma dt - \left\| \overline{u} - u_R^\varepsilon \right\|_\infty \sqrt{\varepsilon C'}    \\
\le & \int_{B_R} \left[ \overline{u}(x,S) - u_R^\varepsilon(x,S) \right] \omega(x) \, dx + \varepsilon \left\| \omega \right\|_\infty V_R \, .
\end{aligned}
$$
By letting $ \varepsilon \downarrow 0 $ in the above inequality, we finally infer that
$$
0 \le \int_{B_R} \left[ \overline{u}(x,S) - u_R(x,S) \right] \omega(x) \, dx \, ,
$$
whence $ \overline{u} \ge u_R $ in $ B_R \times (0,T) $ given the arbitrariness of $ S $ and $ \omega $.

\end{proof}

There hold similar comparison results for nonnegative, bounded solutions to the Dirichlet problem
\begin{equation}\label{pb-balls-app}
\begin{cases}
\Delta W_R = \phi^{-1} \! \left( W_R \right)  & \text{in } B_R \, , \\
W_R = h  &  \text{on } \partial B_R \, , \\
\end{cases}
\end{equation}
where $ h $ is a nonnegative constant. In this case, since $ \phi^{-1} $ is a $ C^1(\mathbb{R}^+) $ function, for simplicity we limit ourselves to dealing with classical solutions (in fact as above we only need $ W_R $ to have a well-defined trace on $ \partial B_R $ in the Lebesgue sense).

\begin{pro}[Comparison for elliptic problems]\label{CE}
Let $ W_R $ be a classical solution of \eqref{pb-balls-app} and $ \overline{W} $ [$ \underline{W} $] be a distributional supersolution [subsolution] to \eqref{elliptic-phi}, in the sense of Definition \ref{vw}. Suppose that $ \overline{W} \ge h $ [$ \underline{W} \le h $] a.e.~in $ B_R $. Then $ \overline{W} \ge W_R $ [$ \underline{W} \le W_R $] a.e.~in $ B_R $.
\end{pro}
\begin{proof}
We shall only give a sketchy proof of the result for subsolutions, since they are more relevant to our purposes (the argument for supersolutions is in any case analogous). Similarly to the proof of Proposition \ref{CP}, first one checks that the inequality
\begin{equation}\label{eq:w-wr}
\int_{B_R} \left[ \left( \underline{W} - W_R \right) \Delta \eta - \left( \phi^{-1}(\underline{W}) - \phi^{-1}(W_R) \right) \eta \right] dx \ge 0
\end{equation}
holds for any nonnegative $ \eta \in C^2(\overline{B}_R) \cap C^1_0(B_R) $, since the boundary term involving $ \frac{\partial \eta}{\partial \nu} $ has the correct sign. Then, given any nonnegative $ \omega \in C_c^\infty(B_R) $, we solve the elliptic problem
\begin{equation*}\label{pb-balls-dual}
\begin{cases}
-\Delta \eta + c(x) \, \eta = \omega  & \text{in } B_R \, , \\
\eta = 0  &  \text{on } \partial B_R \, , \\
\end{cases}
\end{equation*}
where
\begin{equation*}\label{n4-ell}
c(x):=
\begin{cases}
\frac{\phi^{-1}(\underline{W}(x)) - \phi^{-1}(W_R(x))}{\underline{W}(x) - W_R(x)} & \text{if } \underline{W}(x) \neq W_R(x) \, , \\
0 & \text{if } \underline{W}(x) = W_R(x) \, ,
\end{cases}
\end{equation*}
is a nonnegative and bounded coefficient. Upon a routine approximation of $c$ by smooth and positive coefficients (similar to the one regarding the coefficient $ a_\varepsilon $ in the proof of Proposition \ref{CP}), if we plug such a test function in \eqref{eq:w-wr} we end up with the inequality
$$
\int_{B_R}  \left( W_R -\underline{W} \right) \omega \, dx \ge 0 \, ,
$$
whence $ W_R \ge \underline{W} $ in $B_R$ given the arbitrariness of $ \omega $.
\end{proof}

\section{Some known conditions for stochastic (in)completeness}\label{cor}

A number of known (sufficient) conditions for stochastic completeness, or incompleteness, to hold are known, see e.g.~\cite{GrigBams}. We state hereafter the consequences on the uniqueness, or nonuniquess, of solutions to the nonlinear evolution \eqref{gen-zero-bdd}, and on existence, or nonexistence, of bounded solutions to \eqref{elliptic-phi}, that follow by combining those results and Theorems \ref{main-teo}--\ref{main-teob}. As concerns uniqueness, or nonuniqueness, of solutions to \eqref{gen-zero-bdd}, these seem to be the very first results valid on general manifolds.

\begin{cor}[Stochastic incompleteness]\label{incomp}
Let $ M $ be a connected, noncompact Riemannian manifold of dimension $N \ge 2$. Assume that at least one of the following conditions holds.
\begin{enumerate}[(a)]
\item $M$ is nonparabolic, i.e.~it admits a minimal, positive Green function $G$, and $G$ satisfies $\int_{U^c}G(x,y)\,{d}y<\infty$ for some compact set $U\subset M$ and some $x\in U$. \label{green}
    \item Let $r\ge r_0$, with $ r_0 $ sufficiently large. Let $M$ have a pole $o$ and define $r=d(o,x)$, where $d(\cdot,\cdot)$ denotes Riemannian distance. Then there exists $\psi:(r_0,+\infty)\to(0,+\infty)$ such that
    \[
    \Delta r\ge(N-1)\frac{\psi'(r)}{\psi(r)}
    \]
    for all $r\ge r_0$, with
    \[
    \int^\infty_{r_0}\frac1{\psi^{N-1}(r)}\int_{r_0}^r\psi^{N-1}(s)\,{d}s \, dr < \infty \, .
    \]\label{laplno}
\item $M$ has a pole $o$ and the sectional curvature $\mathrm{Sec}_\omega$ w.r.t.~planes containing the radial direction w.r.t.~$o$ satisfies
\[
\mathrm{Sec}_\omega(x)\le -\frac{\psi''(r)}{\psi(r)} \qquad \mathrm{with}\ \, \int^\infty_{0}\frac1{\psi^{N-1}(r)}\int_{0}^r\psi^{N-1}(s)\,{d}s \, dr < +\infty \, ,
\]
where $\psi:[0,+\infty)\to[0,+\infty)$ is strictly positive for $r>0$ with $\psi(0)=0$ and $\psi'(0)=1$.\label{sec}
\item Let $M$ have a pole $o$ and let $r=d(o,x)$ be large enough. There exists a positive increasing smooth function $k(r)$ on $(0,+\infty)$ s.t.~$\int^\infty_{r_0}\frac{{\rm d}r}{k(r)}<+\infty$ for a suitable $r_0>0$, $k'(r)\le Ck^2(r)$ and, again for $r$ large enough,
\[
\mathrm{Sec}_\omega(x)\le-k^2(r) \, .
\]\label{curv2}
\end{enumerate}
Then for all nonnegative bounded data $u_0$, all function $ \phi \in \mathfrak{C} $ and all $ T \in (0,+\infty] $ the Cauchy problem \eqref{gen-zero-bdd} admits at least two nonnegative solutions in $ L^\infty(M\times(0,T)) $. Besides, for all functions $ \phi \in \mathfrak{C} $ the semilinear elliptic equation \eqref{elliptic-phi} admits a nonnegative, nontrivial, bounded solution.

\end{cor}

\smallskip

\begin{cor}[Stochastic completeness]\label{compl}
Let $ M $ be a connected, noncompact Riemannian manifold of dimension $N \ge 2$. Suppose that at least one of the following conditions holds.
\begin{enumerate}[(a)]\setcounter{enumi}{4}
\item $M$ is parabolic.\label{parabolic}
\item For some $o\in M$ the function $ r \mapsto \frac{r}{\log V(o,r)}$ is not integrable at infinity, where $V(o,r)$ is the volume of the geodesic ball of radius $r$ centered at $o$. Note that this holds in particular if $V(o,r)\le Ce^{ar^2}$ for suitable $C,a>0$.\label{vol}
    \item Let $r\ge r_0$, with $ r_0 $ sufficiently large. Let $M$ have a pole $o$ and define $r=d(o,x)$, where $d(\cdot,\cdot)$ denotes Riemannian distance. Then there exists $\psi:(r_0,+\infty)\to(0,+\infty)$ such that
    \[
    \Delta r\le(N-1)\frac{\psi'(r)}{\psi(r)}
    \]
    for all $r\ge r_0$, with
    \[
    \int^\infty_{r_0}\frac1{\psi^{N-1}(r)}\int_{r_0}^r\psi^{N-1}(s)\,{ d}s \, dr =+\infty \, .
    \]\label{lapl}
\item $M$ has a pole $o$ and the Ricci curvature $\mathrm{Ric}_o$ in the radial direction w.r.t.~$o$ satisfies
\[
\mathrm{Ric}_o(x)\ge -(N-1)\frac{\psi''(r)}{\psi(r)} \qquad \mathrm{with }\ \, \int^\infty_{0}\frac1{\psi^{N-1}(r)}\int_{0}^r\psi^{N-1}(s)\,{ d}s \, dr =\infty \, ,
\]
where $\psi:[0,+\infty)\to[0,+\infty)$ is strictly positive for $r>0$ with $\psi(0)=0$ and $\psi'(0)=1$.\label{ric}
\item Let $M$ have a pole $o$ and let $r=d(o,x)$ be large enough. There exists a positive increasing smooth function $k(r)$ on $(0,+\infty)$ s.t.~$\int^\infty_{r_0}\frac{{\rm d}r}{k(r)}=+\infty$ for a suitable $r_0>0$ and, again for $r$ large enough,
\[
\mathrm{Ric}_o(x)\ge-(N-1)k^2(r) \, .
\]\label{curv}
    \end{enumerate}
Then for all nonnegative bounded data $u_0$, all function $ \phi \in \mathfrak{C} $ and all $T>0$ the Cauchy problem \eqref{gen-zero-bdd} admits a unique nonnegative solution in $ L^\infty(M\times(0,T)) $. Besides, for all functions $ \phi \in \mathfrak{C} $ the semilinear elliptic equation \eqref{elliptic-phi} does not admit any nonnegative, nontrivial, bounded solution.

\end{cor}

Note that the conditions in items $(b)$--$(d)$ and $(g)$--$(i)$ amount qualitatively to requiring that (suitable) curvatures tend to minus infinity faster than quadratically at infinity, or not more than quadratically, respectively. Moreover, we point out that the integral conditions in $(c)$ and $(h)$ are \emph{sharp} for model manifolds.
\medskip

\noindent \bf Proofs of Corollaries \ref{incomp} and \ref{compl}.\rm \ The proofs consist in noting that the stated conditions are known to imply stochastic incompleteness, or completeness, respectively. We shall provide references to the statements given in \cite{GrigBams}. In that paper one can also find detailed references to the papers in which some of the results have been proved originally. In fact we see that stochastic \it incompleteness \rm holds under any of the conditions $(a)$--$(d)$ by Corollary 6.7, Corollary 15.2(d), Theorem 15.3(d), Theorem 15.4(b) of \cite{GrigBams}, respectively. Besides, stochastic \it completeness \rm holds under any of the conditions $(e)$--$(i)$ by Corollary 6.4, Theorem 9.1, Corollary 15.2(c), Theorem 15.3(c), Theorem 15.4(a) of \cite{GrigBams}, respectively. \hfill$\qed$

\medskip

\noindent {\sc Acknowledgment.} We thank the anonimous referee for his/her careful reading of our manuscript and useful comments.  K.I.~was partially supported by the Grant-in-Aid for Scientific Research (S) (No. 19H05599)
from Japan Society for the Promotion of Science. G.G.~was partially supported by the PRIN Project {``Equazioni alle derivate parziali di tipo ellittico e parabolico: aspetti geometrici, disuguaglianze collegate, e applicazioni''} (Italy). M.M.~was partially supported by the GNAMPA Project ``Equazioni diffusive non-lineari in contesti non-Euclidei e disuguaglianze funzionali associate'' (Italy). Both G.G.~and M.M.~have also been supported by the GNAMPA group of the Istituto Nazionale di Alta Matematica (INdAM, Italy). M.M.~is grateful to the Tohoku University at Sendai (Japan) for the kind hospitality, since part of this project was carried out during his visit.

\end{document}